\documentclass[12pt]{amsart}
\usepackage{amssymb}
\title{Totally aspherical parameters for Cherednik algebras}
\author{Ivan Losev}
\address{Department
of Mathematics, Northeastern University, Boston MA 02115 USA}
\email{i.loseu@neu.edu}
\newcommand{\h}{\mathfrak{h}}
\renewcommand{\a}{\mathfrak{a}}
\newcommand{\OCat}{\mathcal{O}}
\newcommand{\HC}{\mathfrak{H}}
\newcommand{\C}{\mathbb{C}}
\newcommand{\KZ}{\operatorname{KZ}}
\newcommand{\Hom}{\operatorname{Hom}}
\newcommand{\End}{\operatorname{End}}
\newcommand{\Irr}{\operatorname{Irr}}
\newcommand{\Z}{\mathbb{Z}}
\newcommand{\quo}{/\!/}
\newcommand{\GL}{\operatorname{GL}}
\newcommand{\A}{\mathcal{A}}
\newcommand{\gr}{\operatorname{gr}}
\newcommand{\g}{\mathfrak{g}}

\newcommand{\param}{\mathfrak{p}}
\newcommand{\VA}{\operatorname{V}}

\newcommand{\red}{/\!/\!/}
\newtheorem{Thm}{Theorem}[section]
\newtheorem{Prop}[Thm]{Proposition}
\newtheorem{Cor}[Thm]{Corollary}
\newtheorem{Lem}[Thm]{Lemma}
\theoremstyle{definition}

\newtheorem{Rem}[Thm]{Remark}

\numberwithin{equation}{section}
\begin{document}
\begin{abstract}
We introduce the notion of a totally aspherical parameter for a Rational Cherednik algebra.
We get an explicit construction of the projective object defining the KZ functor for such
parameters. We establish the existence of sufficiently many totally aspherical
parameters for the groups $G(\ell,1,n)$.
\end{abstract}
\maketitle
\section{Introduction}
Let $W$ be a complex reflection group and $\h$ be its reflection representation.
The Rational Cherednik algebra $H_c(W)$ is a filtered deformation of the skew-group
ring $S(\h\oplus\h^*)\#W$ depending on a parameter $c$ that is a collection of complex numbers.
Inside $H_c(W)$, there is  a so called {\it spherical subalgebra} $eH_c(W)e$, where $e\in \C W$
is the averaging idempotent, that deforms $S(\h\oplus \h^*)^W$. The parameter $c$ is called
{\it spherical} if $H_c(W)$ and $e H_c(W) e$ are Morita equivalent (via the bimodule $H_c(W)e$).
The parameter $c$ is called {\it aspherical} if it is not spherical. The goal of this paper
is to investigate parameters that are ``as aspherical as possible''.

The algebra $H_c(W)$ has a triangular decomposition, $H_c(W)=S(\h^*)\otimes \C W\otimes S(\h)$
that is analogous to the triangular decomposition of the universal enveloping algebra of
a semisimple Lie algebra. Thanks to this decomposition, it makes sense to consider  the category $\OCat$
for $H_c(W)$: by definition, the category $\OCat_c(W)$ consists of all $H_c(W)$-modules
$M$ that are finitely generated over $S(\h^*)$ and have locally nilpotent action of
$\h$.

We say that $c$ is {\it totally aspherical}
if, for $M\in \OCat_c(W)$, the following two conditions are equivalent:
\begin{itemize}
\item[(i)] $M$ is torsion as an $S(\h^*)$-module.
\item[(ii)] $eM=0$.
\end{itemize}
We will see below, Proposition \ref{Prop:simplicity},
that $c$ is totally aspherical if and only if the algebra $eH_c(W)e$ is simple.

Let us point out that (ii) always implies (i). So, informally, $c$ is totally  aspherical
if the multiplication by $e$ kills as many modules in $\OCat_c(W)$ as it theoretically can. On the other hand,
there are totally aspherical parameters that are spherical. Those are the parameters
where there are no  $S(\h^*)$-torsion modules in $\OCat_c(W)$. But this case is not interesting:
the category $\OCat_c(W)$ is semisimple.

For example, we can take $W=\mathfrak{S}_n$ (the symmetric group). Here $c$ is a single complex number
and we will show that all parameters in the interval $(-1,0)$ are totally aspherical. More generally,
we will establish the existence of sufficiently many (in a suitable sense) of totally aspherical
parameters for $W=G(\ell,1,n):=\mathfrak{S}_n\ltimes (\Z/\ell\Z)^n$.

We apply the notion of the total asphericity to relate two remarkable modules in $\OCat_c(W)$:
the projective object $P_{KZ}$ and what we call the Harish-Chandra module.

One can construct objects in $\OCat_c(W)$ as follows. We have the induction functor
$\Delta_c: S(\h)\#W\operatorname{-mod}_{fd,ln}\rightarrow \OCat_c(W)$ from the category
of the finite dimensional $S(\h)\#W$-modules with locally nilpotent $\h$-action, $\Delta_c(M):= H_c(W)\otimes_{S(\h)\#W}M$.
For example, take an irreducible $W$-module $\lambda$, it can be viewed as an $S(\h)\#W$-module
with zero action of $\h$. The object $\Delta_c(\lambda)\in \OCat_c(W)$ is called a
Verma module. It has a unique simple quotient to be denoted by $L_c(\lambda)$; the objects
$L_c(\lambda)$ form a complete collection of simples in $\OCat_c(W)$.
Another interesting example is for $M=S(\h)/(S(\h)^W_+)$, where $S(\h)^W_+$
is the ideal in $S(\h)^W$ of all polynomials without constant term and $(S(\h)^W_+)$ stands
for the ideal in $S(\h)$ generated by $S(\h)^W_+$. Recall that $M$ is a graded module
isomorphic to the regular representation of $W$. The module $\HC_c:=\Delta_c(M)$
will be called the {\it Harish-Chandra module}, by analogy with a D-module
on a semisimple Lie algebra, see \cite{HoKas}. A question one can ask is whether
this natural module has any nice categorical properties, for example,
whether it is projective.

On the other hand, in \cite{GGOR}, the authors introduced a crucial tool to study
the category $\OCat_c(W)$, the KZ (Knizhnik-Zamolodchikov) functor. This is a functor
$\KZ:\OCat_c(W)\rightarrow \mathcal{H}_q(W)\operatorname{-mod}$, where
we write $\mathcal{H}_q(W)$ for the Hecke algebra of the group $W$ corresponding
to a parameter $q$ computed from $c$ (we do not need a precise formula). The functor $\KZ$ is exact, on the level
of vector spaces, to a module $M\in \OCat_c(W)$ this functor assigns the fiber
$M_x$ of $M$ at a generic point $x\in \h$. So $\KZ$ is given by $\Hom_{\OCat_c(W)}(P_{\KZ,c},\bullet)$,
where $P_{\KZ,c}$  is a projective object equipped with an epimorphism $\mathcal{H}_q(W)\twoheadrightarrow
\End_{\OCat_c(W)}(P_{\KZ,c})^{opp}$ (which is an isomorphism if and only if $\dim \mathcal{H}_q(W)=|W|$).
The object $P_{\KZ,c}$ is decomposed as
\begin{equation}\label{eq:PKZ_decomp}\bigoplus_{\lambda\in \Irr(W)}P_c(\lambda)^{\mathsf{rk}L_c(\lambda)}.\end{equation}
Here we write $P_c(\lambda)$ for the projective cover of $L_c(\lambda)$ and $\mathsf{rk}$
stands for the generic rank over $S(\h^*)$ (=the dimension of a general fiber).
The object $P_{KZ,c}$ has very nice categorical properties, for example, its irreducible
summands are precisely the indecomposable projectives that are also injectives. On the other
hand, it is difficult to construct $P_{KZ,c}$ explicitly.

Around 2005, Ginzburg, and, independently, Rouquier asked the question when $P_{KZ,c}$ is isomorphic to $\mathfrak{H}_c$
(unpublished).
In this paper, we establish the following criterium for
$P_{KZ,c}\cong \mathfrak{H}_c$.

\begin{Thm}\label{Thm:PKZ_vs_HC}
The following two conditions are equivalent:
\begin{itemize}
\item[(a)] $P_{KZ,c}\cong \HC_c$.
\item[(b)] $c$ is totally aspherical.
\end{itemize}
\end{Thm}

\begin{Cor}\label{Cor:type_A}
For $W=\mathfrak{S}_n$, we have $P_{KZ,c}\cong \HC_c$ for $c\in (-1,0)$.
\end{Cor}

This paper is organized as follows. We start Section \ref{S_asph} by recalling some generalities on
the Rational Cherednik algebras and their categories $\mathcal{O}$. Then we investigate some properties
of $\OCat_c(W)$ for a totally aspherical parameter $c$ and use those properties to prove Theorem \ref{Thm:PKZ_vs_HC}. We finish the
section by recalling a general conjecture on the locus, where the equivalent conditions of Theorem
\ref{Thm:PKZ_vs_HC} hold.

In Section \ref{S_cycl} we study the case of the groups $W=G(\ell,1,n)$. We first prove
Corollary \ref{Cor:type_A}. Then we extend it to the groups $G(\ell,1,n)$ for $\ell>1$, Theorem \ref{Thm:cycl_main}.
The proof of this theorem is based on results from \cite{BL}.

{\bf Acknowledgements}. I would like to thank Raphael Rouquier who explained me a conjectural isomorphism
in Theorem \ref{Thm:PKZ_vs_HC}. I am also grateful  to  Gwyn Bellamy, Roman Bezrukavnikov, Pavel Etingof,
Victor Ginzburg and Iain Gordon for stimulating discussions. This work was supported by the NSF  under Grant  DMS-1161584.

\section{$P_{\KZ}$ vs Harish-Chandra module}\label{S_asph}
\subsection{Reminder on Cherednik algebras}
Let $W$ be a complex reflection group and $\h$  its reflection representation.
We write $S$ for the set of reflections in $W$. For $s\in S$ we choose
elements $\alpha_s\in \h^*, \alpha_s^\vee\in \h$ such that $s\alpha_s=\lambda_s s,
s\alpha_s^\vee=\lambda_s^{-1}\alpha_s^\vee$ with $\lambda_s\in \C\setminus \{1\}$
and $\langle \alpha_s,\alpha_s^\vee\rangle=2$.

Let $c:S\rightarrow \C$
be a function constant on the conjugacy classes.
By definition, \cite[1.4]{EG}, \cite[3.1]{GGOR},
the Rational Cherednik algebra  $H_c(=H_c(W)=H_c(W,\h))$ is the quotient of the skew-group algebra $T(\h\oplus \h^*)\# W$
by the following relations:
$$[x,x']=[y,y']=0,  [y,x]=\langle y,x\rangle-\sum_{s\in S}c(s)\langle x,\alpha_s^\vee\rangle\langle y,\alpha_s\rangle s, \quad x,x'\in \h^*, y,y'\in \h.$$

Let us recall some structural results about $H_c$.
The algebra $H_c$ is filtered (say, with $\deg \h=\deg \h^*=1, \deg W=0$), and its associated graded is $S(\h\oplus \h^*)\#W$, \cite[1.2]{EG}. This yields
the triangular decomposition $H_c=S(\h^*)\otimes \C W\otimes S(\h)$, \cite[3.2]{GGOR}.

Consider
the element $\delta:=\prod_{s}\alpha_s^{\ell_s}\in S(\h^*)^W$, where $\ell_s$
stands for the order of $s$ so that the element $\delta$ is $W$-invariant. Since $\operatorname{ad}\delta$
is locally nilpotent, the quotient $H_c[\delta^{-1}]$ is well-defined. There is a natural isomorphism
$H_c[\delta^{-1}]\cong D(\h^{reg})\#W$, \cite[1.4]{EG},\cite[5.1]{GGOR}.

Consider the averaging idempotent $e:=|W|^{-1}\sum_{w\in W}w\in \C W\subset H_c$. The {\it spherical subalgebra} of $H_c$,
by definition, is $eH_ce$. When the algebras $eH_ce$ and $H_c$
are Morita equivalent (automatically, via the bimodule $H_ce$), we say that the parameter $c$ is {\it spherical}.
Otherwise, we say that $c$ is aspherical.

There is an {\it Euler element} $h\in H_c$ satisfying $[h,x]=x, [h,y]=-y, [h,w]=0$. So the operator $[h,\cdot]$
defines a grading on $H_c$ to be called the Euler grading. The Euler element is constructed as follows.
Pick a basis $y_1,\ldots,y_n\in \h$ and let $x_1,\ldots,x_n\in \h^*$ be the dual basis.  Then
\begin{equation}\label{eq:Euler}
h=\sum_{i=1}^n x_i y_i+\frac{n}{2}-\sum_{s\in S} \frac{2c(s)}{1-\lambda_s}s.\end{equation}

\subsection{Reminder on categories $\mathcal{O}$}\label{SS_cat_O}
Following \cite[3.2]{GGOR}, we consider the full subcategory $\OCat_c(W)$ of $H_c\operatorname{-mod}$
consisting of all modules $M$ that are finitely generated over $S(\h^*)$ and with locally nilpotent
action of $\h$. Equivalently, we can require that the modules in $\OCat_c(W)$ are finitely generated
over $H_c$ (and $\h$ still acts locally nilpotently). In the category $\OCat_c(W)$ we have analogs of Verma modules that
are indexed by the  irreducible representations $\lambda$ of $W$. By definition, the Verma module indexed
 by $\lambda$ is $\Delta_c(\lambda):=H_c\otimes_{S(\h)\#W}\lambda$, where $\h$ acts by $0$ on $\lambda$. It is easy to see that
this module  is in $\mathcal{O}_c(W)$. Using the Euler element, one shows that $\Delta_c(\lambda)$ has a unique simple
quotient to be denoted by $L_c(\lambda)$. The objects $L_c(\lambda)$ form a complete collection of simples
in $\OCat_c(W)$.

Often we drop $W$ from the notation and just write $\OCat_c$.
Let us note that all finite dimensional modules lie in $\OCat_c$ thanks to the presence of the Euler
element.

To a module  $M\in\OCat_c$ we can assign its associated variety $\VA(M)$ that, by definition, is the support of $M$
(as a coherent sheaf)  in $\h$. Clearly, $\VA(M)$ is a closed $W$-stable subvariety. Also
to a module $M$ in $\OCat_c$, one can assign its  rank, $\mathsf{rk}M$, the generic rank of
$M$ viewed as an  $S(\h^*)$-module.

The category $\OCat_c$ has enough projective objects, see, for example, \cite[Theorem 2.19]{GGOR}.
For $\lambda\in \operatorname{Irr}W$, let $P_c(\lambda)$ denote the projective
cover of $L_c(\lambda)$.  Following \cite{GGOR}, we consider the projective object $P_{\KZ}$ defined by
$$P_{\KZ}=\bigoplus_{\lambda\in \operatorname{Irr}W}P_c(\lambda)^{\mathsf{rk}L_c(\lambda)}.$$
We write $P_{\KZ,c}$ when we want to indicate the parameter.
The importance of this projective is that it defines a quotient functor from $\OCat_c$
to the category of modules over a suitable Hecke algebra.

We will also need the category $\OCat$ over $eH_ce$ (to be denoted by $\OCat_c^{sph}$), see \cite[Section 3]{GL}. By definition, this category consists of all finitely generated $eH_ce$-modules $N$ such that
\begin{enumerate}
\item $ehe=(=he=eh)$ acts on $N$ locally finitely.
\item the sum of the positive graded components of $eH_ce$ (with respect to the Euler grading) acts locally
nilpotently on $N$.
\end{enumerate}

The following lemma is a consequence of \cite[Proposition 3.2.1]{GL}.

\begin{Lem}\label{Lem:quot_to_spher}
The functor $M\mapsto eM$ is a quotient functor $\OCat_c\rightarrow \OCat_c^{sph}$. The left adjoint and right inverse
is given by $N\mapsto H_ce\otimes_{eH_ce}N$.
\end{Lem}

For $N\in \OCat_c^{sph}$, we define the generic rank, $\mathsf{rk}N$, as the generic rank of $N$
viewed as an $S(\h^*)^W$-module. Note that $\mathsf{rk} N=\mathsf{rk}\, eN$.

We have  naive duality functors $\OCat_c(W,\h)\xrightarrow{\sim} \OCat_{c^*}(W,\h^*), \OCat^{sph}_c(W,\h)
\xrightarrow{\sim}\OCat_{c^*}^{sph}(W,\h^*)$, where $c^*(s):=c(s^{-1})$, see \cite[4.2]{GGOR} for the former.
Namely, consider the isomorphism $\varphi: H_c(W,\h)\rightarrow H_{c^*}(W,\h^*)^{opp}$ that is
the identity on $\h,\h^*$ and maps $w$ to $w^{-1}$, $w\in W$. Let us note that $\varphi$ restricts
to an isomorphism $e H_c(W,\h)e\xrightarrow{\sim} (e H_{c^*}(W,\h^*)e)^{opp}$. For $M\in \OCat_c(W,\h)$,
the restricted dual $M^\vee$ (the sum of the duals of the generalized eigenspaces for $h$)
is a left $H_{c^*}(W,\h^*)$-module that lies in the category $\OCat$. The functor $M\mapsto M^\vee$
is the duality functor that we need. A duality between the spherical categories is defined in a similar
fashion.

\begin{Lem}\label{Lem:GK_preserv}
The duality $\OCat_c(W,\h)\xrightarrow{\sim} \OCat_{c^*}(W,\h^*)$ preserves the GK dimensions
and generic ranks. A similar claim holds for the spherical categories $\OCat$.
\end{Lem}
\begin{proof}
We can read the GK dimension and the GK multiplicity of a graded $S(\h^*)$-module from the dimensions
of the graded components (by the definitions of those invariants). Note that the GK multiplicity is the same
as the generic rank.

Choose a section $\sigma: \C/\mathbb{Z}\rightarrow \C$. Introduce the grading on
$M\in \OCat_c$ by declaring that the degree of the generalized eigenspace for $h$ with eigenvalue
$\alpha$ is $\alpha-\sigma(\alpha)$. Then we get the grading on $M^\vee$ compatible with the Euler
grading on $H_{c^*}(W,\h^*)$ and such that $\dim M^\vee_n= \dim M_n$. This completes the proof for the usual
categories $\OCat$. The proof for spherical categories is similar.
\end{proof}


We will also need induction and restriction functors for categories $\OCat_c$ introduced in \cite{BE}.
Let $W'$ be a parabolic subgroup of $W$. Then we have an exact functor $\operatorname{Res}_W^{W'}:
\mathcal{O}_c(W)\rightarrow \mathcal{O}_c(W')$. The target category is the category $\mathcal{O}$
for $H_c(W')$, where the parameter is obtained by restricting $c$ to $W'\cap S$. We will need two
facts about the restriction functor:
\begin{itemize}
\item If $W'$ is the stabilizer of a generic point in an irreducible component of $\VA(M)$, then
$\operatorname{Res}_W^{W'}(M)$ is a nonzero finite dimensional module.
\item The functor $\operatorname{Res}_W^{W'}$ induces a functor between the quotients,
$\OCat_c^{sph}\rightarrow \OCat^{sph}_c(W')$.
\end{itemize}
(a) was established in \cite{BE}, while (b) is in \cite[3.5]{GL}.

\subsection{Totally aspherical parameters}
Recall that we call a parameter $c$ {\it totally aspherical} if $eM=0$ for all $S(\h^*)$-torsion modules
$M\in \OCat_c(W)$. Equivalently, $c$ is totally aspherical if and only if $\VA(N)=\h/W$ for all $N\in \OCat_c^{sph}(W)$.

\begin{Lem}\label{Lem:freeness}
If $c$ is totally aspherical, then all modules in $\OCat_c^{sph}(W)$ are free over $S(\h^*)^W$.
\end{Lem}
\begin{proof}
Repeating the argument from \cite[3.2]{EGL}, we see that all modules in $\OCat_c^{sph}(W)$ are Cohen-Macaulay
over $S(\h^*)^W$. Since all modules are torsion-free, this implies that they are projective. But they are
also graded and hence are free.
\end{proof}

\begin{Cor}\label{Cor:invar}
Let $c$ be a totally aspherical parameter and let $M\in \OCat_c^{sph}(W)$.
Then $$\dim M^{S(\h)^W_+}=\mathsf{rk}M.$$
\end{Cor}
\begin{proof}
Recall the duality $\bullet^\vee: \OCat_{c}^{sph}(W,\h)\xrightarrow{\sim} \OCat_{c^*}^{sph}(W,\h^*)$
explained in Subsection \ref{SS_cat_O}.
The parameter $c^*$ is totally aspherical
for $H_{c^*}(W,\h^*)$, for example, because the duality preserves GK dimensions, Lemma \ref{Lem:GK_preserv}.
So $\dim M^\vee/ S(\h)^W_+ M^\vee=\mathsf{rk}M^\vee=\mathsf{rk}M$, the first equality follows from
Lemma \ref{Lem:freeness} and the second one from Lemma \ref{Lem:GK_preserv}.
But $\dim M^{S(\h)^W_+}=\dim M^\vee/ S(\h)^W_+ M^\vee$ by the construction of $\bullet^\vee$.
\end{proof}

\subsection{Harish-Chandra module vs $P_{KZ}$}
In this subsection, we prove Theorem \ref{Thm:PKZ_vs_HC}.

For $M\in \OCat_c(W)$, we have \begin{equation}\label{eq:Hom_dim}\Hom_{\OCat_c(W)}(\HC_c, M)\cong \Hom_{S(\h)\#W}(S(\h)/(S(\h)^W_+),M)\cong   e(M^{S(\h)^W_+})=(eM)^{S(\h)^W_+}.\end{equation}

Let us prove the implication $(b)\Rightarrow (a)$.
It follows from Corollary \ref{Cor:invar} that $$\dim \Hom_{\OCat_c(W)}(\HC_c,M)=\mathsf{rk}(eM).$$
The right hand side coincides with  $\mathsf{rk}(M)$ which, in turn,
equals $\dim \Hom_{\OCat_c(W)}(P_{\KZ,c},M)$. So we conclude that
\begin{equation}\label{eq:Hom_dim2}\dim\Hom_{\OCat_c(W)}(\HC_c,M)=\dim \Hom_{\OCat_c(W)}(P_{\KZ,c},M).\end{equation}
Since $P_{\KZ,c}$ is projective, equivalently, the functor $\Hom_{\OCat_c(W)}(P_{\KZ,c},\bullet)$ is exact,
we see that the functor $\Hom_{\OCat_c(W)}(\HC_c,\bullet)$ is exact, equivalently, $\HC_c$ is projective.
(\ref{eq:Hom_dim2}) now implies  $\HC_c\cong P_{\KZ,c}$.

Let us prove (a)$\Rightarrow$(b). Assume there is a simple object $M\in \OCat_c(W)$ that is $S(\h^*)$-torsion
(equivalently, $\Hom_{\OCat_c(W)}(P_{\KZ,c},M)=0$) such that $eM\neq 0$. Since $S(\h)^W_+$ acts
locally nilpotently, we see that $(eM)^{S(\h)^W_+}\neq 0$. Thanks to (\ref{eq:Hom_dim}),
we see that $\Hom_{\OCat_c(W)}(\HC_c,M)\neq 0$. So $\HC_c\not\cong P_{\KZ,c}$.

\begin{Rem}\label{Rem:not_proj}
Note that if $c$ is not totally aspherical, then $\mathfrak{H}_c$ is not projective. Assume the converse. We have
\begin{align*}&\dim \operatorname{Hom}_{\OCat_c(W)}(\mathfrak{H}_c, M)=\dim (e M)^{S(\h)^W_+}=
\dim (eM)^\vee/ S(\h)^W_+ (eM)^\vee\\&\geqslant
\mathsf{rk}(eM)^\vee=\mathsf{rk}(eM)=\mathsf{rk}(M).\end{align*}
So, for a torsion free $L_c(\lambda)$, we see that the multiplicity of $P_c(\lambda)$ in $\mathfrak{H}_c$
is not less than the multiplicity of $P_c(\lambda)$ in $P_{\operatorname{KZ},c}$. But both $P_{\operatorname{KZ},c}$
and $\mathfrak{H}_c$ have a filtration of length $\sum_{\lambda\in \operatorname{Irr}W}\dim \lambda$ with
Verma quotients. This length is independent of the choice of the filtration and hence
$\mathfrak{H}_c\cong P_{\operatorname{KZ},c}$. Contradiction.
\end{Rem}


\subsection{Total asphericity vs simplicity}
Here we are going to provide one more equivalent formulation of total asphericity.

\begin{Prop}\label{Prop:simplicity}
A parameter $c$ is totally aspherical if and only if the algebra $eH_ce$ is simple.
\end{Prop}
 Let us introduce the following condition
on $c$.

\begin{itemize}
\item[(*)] $e_{W'}M'=0$ for any parabolic subgroup
$W'\subset W$ different from $\{1\}$ and any finite dimensional $H_c(W')$-module $M'$. Here we write
$e_{W'}$ for the averaging idempotent in $\C W'$.
\end{itemize}

The proposition follows from the next two lemmas.

\begin{Lem}\label{Lem:tot_asph_char}
The parameter $c$ is totally aspherical  if and only if (*) holds.
\end{Lem}

This lemma will also be important in Section \ref{S_cycl} to determine  aspherical values for
the groups $W=G(\ell,1,n)$.

\begin{proof}
Recall  the restriction functors $\operatorname{Res}_W^{W'}:
\OCat_c(W)\rightarrow \OCat_c(W'),\OCat_c^{sph}(W)\rightarrow \OCat^{sph}_{c}(W')$ and  that
for any nonzero $M\in \OCat_c(W)$ there is $W'$ such that $\operatorname{Res}_W^{W'}(M)$
is nonzero finite dimensional. So if $e_{W'}M'=0$ for any finite dimensional $M'\in \OCat_c(W')$,
then $c$ is totally aspherical for $H_c(W)$.

Now suppose that $c$ is totally aspherical. Recall the functor $\operatorname{Ind}_{W'}^W$ from
\cite{BE}, a left and a right adjoint of $\operatorname{Res}_W^{W'}$. Let $N$ be a finite dimensional
$e_{W'}H_c(W')e_{W'}$-module. The $H_c(W')$ module
$\tilde{N}:=H_c(W')e_{W'}\otimes_{e_{W'}H_c(W')e_{W'}}N$ is also finite dimensional because
$H_c(W')e$ is finitely generated over $e_{W'}H_c(W')e_{W'}$. The module $\operatorname{Ind}_{W'}^W(\tilde{N})$
is nonzero and has proper support by \cite[Proposition 2.7]{SV}. So $e\operatorname{Ind}_{W'}^W(\tilde{N})=0$. The module
$\operatorname{Res}_W^{W'}\circ\operatorname{Ind}_{W'}^W(\tilde{N})$ is not killed by $e_{W'}$ because it admits
a nonzero homomorphism to $\tilde{N}$. On the other hand,
$$e_{W'}\operatorname{Res}_W^{W'}\circ\operatorname{Ind}_{W'}^W(\tilde{N})=
\operatorname{Res}_W^{W'}(e\operatorname{Ind}_{W'}^W(\tilde{N}))=0.$$
This contradiction completes the proof.
\end{proof}


\begin{Lem}\label{Lem:simpl_char}
The algebra $eH_ce$ is simple if and only if (*) holds.
\end{Lem}
\begin{proof}
Let $\mathcal{J}\subset eH_ce$ be a proper two-sided ideal. Its associated variety in $(\h\oplus \h^*)/W$
is proper. Pick a generic point (with stabilizer, say, $W'$) in an irreducible component and consider the functor $\bullet_{\dagger}$
from \cite{sraco} at that  point. Then $\mathcal{J}_{\dagger}$ is a proper ideal of finite codimension in
$e_{W'} H_c(W')e_{W'}$, by a direct analog of \cite[Theorem 3.4.6]{sraco} for spherical subalgebras
(that holds with the same proof). So $c$ is not totally aspherical.

Now assume that $e_{W'}H_c(W')e_{W'}$ has a finite dimensional representation. Let $\mathcal{I}$ be the annihilator
and set $\mathcal{B}:=e_{W'}H_c(W')e_{W'}/\mathcal{I}$. Then we apply the functor $\bullet^{\dagger}$ associated to $W'$
from  \cite{sraco} to $\mathcal{B}$ and get a Harish-Chandra $H_c$-bimodule
$\mathcal{B}^{\dagger}$ with a homomorphism $eH_c(W)e\rightarrow \mathcal{B}^{\dagger}$. The kernel is
a proper two-sided ideal, its associated variety is the closure of the locus in $(\h\oplus\h^*)/W$
with stabilizer $W'$. So $eH_ce$ is  not simple.
\end{proof}


\subsection{Shifts}\label{SS_shifts}
There is a conjecture saying that there are sufficiently many totally aspherical parameters.
Namely, let us write $\param$ for the space of the $W$-invariant functions $c:S\rightarrow \C$.
Let us write $c_s$ for a function that maps $s'\in S$ to $1$ if $s'$ is conjugate to $s$
and to $0$ else.  For a reflection hyperplane $H$ and $i\in \{0,\ldots,\ell_H-1\}$ (where
$\ell_H$ is the cardinality of the pointwise stabilizer $W_H$) define an element $h_{H,i}\in
\param^*$  by
\begin{equation}\label{eq:h_to_c} h_{H,i}(c)=\frac{1}{\ell_H}\sum_{s\in W_H\setminus \{1\}}\frac{2c_s}{\lambda_s-1}\lambda_s^{-i}
\end{equation}
The elements $h_{H,i}$ with $i\neq 0$ form a basis in $\param^*$ and $\sum_{i=0}^{\ell_H-1}h_{H,i}=0$.
Let $\param^*_{\Z}$ denote the integral
lattice in $\param^*$ spanned by the elements $h_{H,i}-h_{H,0}$
and let $\param_{\Z}$ be dual to $\param^*_\Z$.
For example, when $W=\mathfrak{S}_n$, we get $h_{H,1}-h_{H,0}=c$ and the lattice
$\param_{\Z}$ consists of integers.

In general, the meaning of the  lattice $\param_{\Z}$ is as follows. Recall that from a Cherednik
parameter $c$ one can produce a parameter $q$ for the Hecke algebra of $W$. Two Cherednik parameters
produce the same Hecke parameter if their difference lies in $\param_{\Z}$.

It was conjectured by Rouquier (unpublished) that every coset $c+\param_{\Z}$
contains a parameter $c$ with $P_{\operatorname{KZ},c}=\mathfrak{H}_c$, i.e., see Theorem
\ref{Thm:PKZ_vs_HC},  a completely aspherical parameter. In the next section, we will prove a theorem that  implies  this conjecture
for the groups  $W=G(\ell,1,n)$.

\section{Cyclotomic case}\label{S_cycl}
\subsection{Type $A$}\label{SS_type_A}
Here we are going to prove Corollary \ref{Cor:type_A}. Our proof is based on Lemma \ref{Lem:tot_asph_char}.

Recall that the algebra $H_c(\mathfrak{S}_m)$
has a finite dimensional  representation if and only if $c$ is rational with denominator $m$,
and in this case, there is a unique finite dimensional irreducible representation and every
finite dimensional representation is completely reducible, see \cite{BEG}. Further, the finite
dimensional irreducible representation is annihilated by $e$ if and only if $c\in (-1,0)$.
All parabolic subgroups of $\mathfrak{S}_n$ are of the form $\mathfrak{S}_{n_1}\times\ldots\times
\mathfrak{S}_{n_k}$ with $\sum_{i=1}^k n_i\leqslant n$. Now Corollary \ref{Cor:type_A}
follows from Lemma \ref{Lem:tot_asph_char}.

Let us point out that the totally aspherical parameters outside $(-1,0)$ are either irrational
or have denominator bigger than $n$ (and so  the category $\mathcal{O}_c(\mathfrak{S}_n)$ is semisimple).


\subsection{Quantized quiver varieties}
We proceed to the case of the groups $G(\ell,1,n)$. The spherical subalgebras in this case can be realized
as quantized quiver varieties (i.e., as Hamiltonian reductions of the algebras of differential operators on
spaces of quiver representations). In this subsection we recall how this is done.

Let $Q=(Q_0,Q_1,t,h)$ be the affine Dynkin quiver with $\ell$ vertices
(and some orientation). Consider the representation space
$R=R(Q,v,\epsilon_0)$ with dimension vector $v=(v_i)_{i\in Q_0}$  and one-dimensional co-framing
at the extending vertex, explicitly, $R=\bigoplus_{a\in Q_1}\Hom(V_{t(a)}, V_{h(a)})\oplus V^*_0$,
where $\dim V_i=v_i$. On the space $R$ we have a natural action of the group $G=
\prod_{i\in Q_0}\GL(V_i)$. Let $\epsilon_i\in \C^{Q_0}$ denote the coordinate vector (=the simple root) at
the vertex $i$ and $\delta=\sum_{i\in Q_0} \epsilon_i$ be the indecomposable imaginary root.

To the action of $G$ on $R$ we can associate several varieties/algebras obtained by Hamiltonian reduction.
We have a moment map $\mu: T^*R\rightarrow \g^*$ defined as follows. Its dual map $\mu^*:\g\rightarrow \C[T^*R]$
sends $x\in \g$ to the velocity vector field $x_R\in \operatorname{Vect}(R)$ viewed as an element
of $\C[T^*R]$ via a natural inclusion $\operatorname{Vect}(R)\hookrightarrow \C[T^*R]$.
Consider the variety $X^0(v):=T^*R\red_0G(=\mu^{-1}(0)\quo G)$. When $v_i=n$ for all $i$ (i.e., $v=n\delta$),
this variety is identified with $(\h\oplus\h^*)/W$, where $W=G(\ell,1,n)$, see, e.g.,
\cite[7.7]{Gordon_review}. The identification can
be made equivariant with respect to the $\C^\times$-actions by dilations: the action on
$\mu^{-1}(0)\quo G$ is induced from the action on $T^*R$ by dilations.

Next, take a generic enough character $\theta$ of the group $G$. We can form the GIT reduction
$X^\theta(v):=T^*R\red_0^\theta G(=\mu^{-1}(0)^{\theta-ss}\quo G)$. This is a smooth symplectic variety
equipped with a natural morphism $X^\theta(v)\rightarrow X^0(v)$.
When $v=n\delta$, this morphism   is a resolution of singularities, see, e.g.,
\cite[7.8]{Gordon_review}.

One can also quantize the variety $X^0(n\delta)$ using the quantum Hamiltonian reduction.
Define the symmetrized quantum comoment map $\Phi(x):=\frac{1}{2}(x_R+x_{R^*}):\g\rightarrow D(R)$, where
$\g$ stands for the Lie algebra of $G$ and $D(R)$ for the algebra of linear differential operators on
$R$. Note that $\Phi(x)$ does not depend on the orientation of $Q$. Pick
a collection $\lambda=(\lambda_0,\ldots,\lambda_{\ell-1})$ of complex numbers
and view it as a character of $\g$ via $\langle \lambda,x\rangle:=\sum_{i=0}^{\ell-1}\lambda_i
\operatorname{tr}(x_i)$. Then set $$\A_\lambda(n\delta):=D(R)\red_\lambda G
(=[D(R)/D(R)\{\Phi(x)-\langle\lambda,x\rangle\}]^G).$$
This is a filtered algebra with $\gr\A_\lambda(n\delta)=\C[X^0(n\delta)]$. Moreover, we have a filtered algebra
isomorphism $\A_\lambda(n\delta)\cong e H_c(W)e$, where one recovers $\lambda$ from $c$ as follows.
We can encode the parameter $c$ as  $(\kappa, c_1,\ldots,c_{\ell-1})$, where $\kappa\in \C$
is the parameter corresponding to a reflection in $\mathfrak{S}_n$ and $c_i$ corresponds
to $i\in \Z/\ell\Z$.

According to \cite{Gordon_cyclic}, see also \cite[6.2]{quant_res}, we get
\begin{equation}\label{eq:param_corresp}
\begin{split}
&\lambda_k=\frac{1}{\ell}(1+2\sum_{j=1}^{\ell-1} c_j \exp(2\pi\sqrt{-1}j/\ell)), k\neq 0,\\
&\lambda_0=\frac{1}{\ell}(1+2\sum_{j=1}^{\ell-1}c_j)+\kappa-\frac{1}{2}.
\end{split}
\end{equation}

We can also produce quantizations of $X^\theta(v)$ using quantum Hamiltonian reduction.
Namely, we microlocalize the algebra $D(R)$  to a sheaf (in the conical topology) $D_R$ of filtered algebras
on $T^*R$. Then we set
$$\A_\lambda^\theta(v):=D_R\red^\theta_\lambda G(=[D_R/D_R\{\Phi(x)-\langle \lambda,x\rangle\}|_{(T^*R)^{\theta-ss}}]^G)$$ (note that this definition differs from \cite{BL}
by a shift of $\lambda$ because here we use the symmetrized quantum comoment map). This is a sheaf of filtered algebras of $X^\theta(v)$ with $\gr\A_\lambda^\theta(v)=\mathcal{O}_{X^\theta(v)}$ and $\Gamma(\A_\lambda^\theta(n\delta))=
\A_\lambda(n\delta)$. In general, we set $\A_\lambda(v):=\Gamma(\A_\lambda^\theta(v))$, the right hand
side does not depend on the choice of $\theta$ by \cite[3.3]{BPW}. We also set $\A^0_\lambda(v)=[D(R)/D(R)\{\Phi(x)-
\langle \lambda,x\rangle\}]^G,  X(v):=\operatorname{Spec}(\C[X^\theta(v)])$. Note that there is a natural
homomorphism $\A^0_\lambda(v)\rightarrow \A_\lambda(v)$. Let $\rho_v$ denote the natural morphism
$X^\theta(v)\rightarrow X(v)$, it is a resolution of singularities for any $v$. We often write $\rho$ instead
of $\rho_v$.



We can consider
the derived global section functor $$R\Gamma:D^b(\A_\lambda^\theta(v)\operatorname{-mod})\rightarrow
D^b(\A_\lambda(v)\operatorname{-mod}).$$ Here we write $\A_\lambda^\theta(v)\operatorname{-mod}$
for the category of all $\A_\lambda^\theta(v)$-modules that are direct limits of coherent modules
(``coherent'' is the same as ``good'' in \cite[4.5]{MN}). Note that $\A_\lambda^\theta(v)\operatorname{-mod}$
is the quotient of the category of the $(G,\lambda)$-equivariant $D(R)$-modules by the Serre
category of all $D$-modules with $\theta$-unstable support.

The functor $R\Gamma$ restricts to
$R\Gamma: D^b_{\rho^{-1}(0)}(\A_\lambda^\theta(v)\operatorname{-mod})\rightarrow
D^b_{fin}(\A_\lambda(v)\operatorname{-mod})$. Here the superscript ``$\rho^{-1}(0)$'' indicates
the category of all complexes with  coherent homology supported on $\rho^{-1}(0)$. Similarly,
the subscript ``fin'' means the category of all complexes with finite dimensional homology.

\subsection{Main result in the cyclotomic case}
Let us state our main result. Define the subalgebra $\a^{\lambda}\subset \g(Q)=\hat{\mathfrak{sl}}_\ell$
as follows. It is generated by the Cartan subalgebra $\mathfrak{t}$ of $\g(Q)$ and the root subspaces
$\g(Q)_\beta$ for real roots $\beta=\sum_{i\in Q_0}b_i \epsilon_i$ such that
$$\sum_{i\in Q_0} b_i (\lambda_i+\frac{1}{2}(w_i+\sum_{a, t(a)=i}v_{h(a)}+\sum_{a, h(a)=i}v_{t(a)}))\in \Z.$$

\begin{Thm}\label{Thm:cycl_main}
Suppose that
\begin{enumerate}
\item
$|\langle \lambda,\beta\rangle|\leqslant \frac{1}{2}|b_0|$ for all roots $\beta=(b_i)_{i\in Q_0}$
of $\a^{\lambda}$,
\item and $\langle \lambda,\delta\rangle-\frac{1}{2}$ is either in $(-1,0)$ or an integer or has denominator bigger
than $n$ (irrational numbers have denominator $+\infty$, by convention).\end{enumerate} Then the parameter
$c$ computed from $\lambda$ using (\ref{eq:param_corresp}) is totally aspherical.
\end{Thm}


\begin{Cor}\label{Rem:cycl_main}
Rouquier's conjecture mentioned in Subsection \ref{SS_shifts} is true for the groups $G(\ell,1,n)$:
every coset $c+\param_\Z$ contains a totally aspherical parameter.
\end{Cor}
\begin{proof}
Note that, under the isomorphism $\param\cong \C^{Q_0}$ given by (\ref{eq:param_corresp}), the lattice
$\param_{\Z}$ coincides with  $\Z^{Q_0}$,  a computation is somewhat implicitly contained,  e.g., in
\cite[2.3.1]{GL}. So  it is enough
to show that every $\lambda\in \C^{Q_0}$ admits an integral shift satisfying the conditions of
Theorem \ref{Thm:cycl_main}.

If $\a^{\lambda}=\mathfrak{t}$, then the proof is easy, we just need to achieve condition (2).
Assume $\a^{\lambda}\neq \mathfrak{t}$.
Choose a system of simple roots, $\beta_1,\ldots,\beta_k$, for $\a^{\lambda}$ consisting of positive roots.
We claim that it is enough to show that there is an element $\lambda'\in \lambda+\Z^{Q_0}$ satisfying
\begin{equation}\label{eq:ineq} 0\leqslant \langle \lambda'+\omega_0^\vee/2, \beta_i\rangle
\leqslant  (\omega_0, \beta_i), \quad \forall i=1,\ldots,k.\end{equation} Here
$\omega_0^\vee$ denotes the fundamental coweight corresponding to $0$ and $(\cdot,\cdot)$ stands for
the usual invariant scalar  product. Note that $\langle \lambda'+\omega_0^\vee/2, \beta_i\rangle$ is an integer.
Clearly, (\ref{eq:ineq}) implies (1). Now suppose the denominator $d$ of $\langle \lambda,\delta\rangle$
satisfies $d\leqslant n$. Then $d\delta$ is integral on $\lambda$ and hence $(\omega_0,\beta_i)\in
\{0,\ldots,d-1\}$.  It follows that $d\delta$ is the sum of some of $\beta_i$ and hence
(\ref{eq:ineq}) also implies (2).

For simplicity, consider the case when $k=\ell$, the general case ($k<\ell$) is analogous.
We can apply an element from $\mathfrak{S}_\ell$ to $\lambda$, this does not change the inequalities
(\ref{eq:ineq}). Thanks to this, we can assume that $\beta_i=\alpha_i+d_i\delta, i=1,\ldots,\ell$.
Set $x_i:=\langle \lambda'+\frac{\omega_0^\vee}{2},\alpha_i\rangle$. Then (\ref{eq:ineq}) can be
rewritten as
\begin{equation}\label{eq:ineq1} 0\leqslant x_i+d_i(x_1+\ldots+x_\ell)\leqslant d_i+ \delta_{i,\ell}.
\end{equation}
Set $s:=x_1+\ldots+x_\ell$ and observe that $1+\sum_{i=1}^\ell d_i=d$. So $s$ and $d$ are coprime integers.
Summing the inequalities (\ref{eq:ineq1}), we get $0\leqslant ds\leqslant d$. This specifies
$s$ uniquely (with the exception of the non-interesting case $d=1$). One can then rewrite
(\ref{eq:ineq1}) as
\begin{equation}\label{eq:ineq2}
\begin{split}
& 0\leqslant x_i+d_is\leqslant d_i,\quad i=1,\ldots,\ell-1,\\
& 0\leqslant ds- \sum_{i=1}^{\ell-1} (x_i+d_is)\leqslant d-\sum_{i=1}^{\ell-1} d_i
\end{split}
\end{equation}
It is now clear that these inequalities have a solution.
\end{proof}


Let us describe key ideas of the proof of Theorem \ref{Thm:cycl_main}. Lemma \ref{Lem:tot_asph_char} reduces the problem to
checking the absence of the finite dimensional representations. To show that $\A_\lambda(n\delta)$
has no finite dimensional representations it is enough to prove that $R\Gamma(M)=0$
for any $\A_\lambda^\theta(n\delta)$-module $M$ supported on $\rho^{-1}_{n\delta}(0)$.
The structure of the irreducible $\A_\lambda^\theta(v)$-modules supported on $\rho^{-1}_v(0)$
was studied in \cite{BL}. We will see that under condition (1), the algebra $\A_\lambda(n\delta)$
has no finite dimensional representations, this is a key step. From here and results of Subsection \ref{SS_type_A}
we will see that (1) and (2) guarantee the absence of the finite dimensional representations
of the slice algebras (which in this case means the spherical subalgebras for the parabolic
subgroups).

Let us discuss our key step in more detail. In \cite{BL} we have shown that the category
$\bigoplus_{v} D^b(\A_{\lambda_v}^\theta(v)\operatorname{-mod}_{\rho_v^{-1}(0)})$ carries
a categorical action (in a weak sense) of $\a^{\lambda}$ (here the parameters $\lambda_v$ depend on $v$
in a suitable way). This means that there are
endo-functors $E_\alpha,F_\alpha$ indexed by simple roots $\alpha$ of $\a^{\lambda}$ that define a
usual representation of $\a^\lambda$ on $K_0$. Each pair $E_\alpha,F_\alpha$ defines a categorical
$\mathfrak{sl}_2$-action in the strong sense (of Rouquier, \cite{Rouquier_2Kac}). Moreover, $E_\alpha,F_\alpha$
define an $\mathfrak{sl}_2$-crystal structure on the set of simples with operators $\tilde{e}_\alpha, \tilde{f}_\alpha$.
We will see that the condition $|\langle \lambda,\alpha\rangle|\leqslant \frac{1}{2}( \omega_0, \alpha)$ guarantees
that we have $R\Gamma(M)=0$ for a simple $M$ lying in the image of $\tilde{f}_\alpha$.
We will first do this when $\alpha$ is a simple root for $\g(Q)$, Subsection \ref{SS_EF_i_func}, and then extend this result to the general case.

\begin{Rem}\label{Rem:more_tot_asph}
It does not seem that Theorem \ref{Thm:cycl_main} describes all totally aspherical parameters: it should not
be true that condition (1) specifies precisely the parameters where $\A_\lambda(n\delta)$ has no finite dimensional
representations. \cite[Conjecture 9.8]{BL} reduces the question of when $\A_\lambda(n\delta)$ has no finite dimensional
representation to a problem in the integrable representation theory of $\mathfrak{a}^\lambda$.
\end{Rem}

\begin{Rem}
In the case when $n=1$ all parameters $c$ where $P_{KZ,c}\cong \mathfrak{H}_c$ were described in
\cite{Thelin}.
\end{Rem}

\subsection{Functors $E_i,F_i$}\label{SS_EF_i_func}
An important role in the proof of the main theorem of \cite{BL} is played by Webster's functors $E_i,F_i$ between
the derived categories of modules over the sheaves $\A_\lambda^\theta(v)$.

Fix $i\in Q_0$ such that $\lambda_i\in \Z+\frac{1}{2}(w_i+\sum_{a,t(a)=i}v_{h(a)}+\sum_{a, h(a)=i}v_{t(a)})$  and
assume that $\theta_k>0$ for all $k\in Q_0$.
We are going to recall the functors
$$F_i: D^b(\A_\lambda^\theta(v)\operatorname{-mod})\rightleftarrows D^b(\A_{\lambda'}^\theta(v+\epsilon_i)\operatorname{-mod}):E_i$$
introduced in \cite{Webster}. Here $\lambda'$ is a parameter recovered from $\lambda$ (below we will
explain how).

Assume $i$ is a source (otherwise we can invert some arrows, this does not change $\lambda$
and the sheaf $\A_\lambda^\theta(v)$). Set $\tilde{w}_i:=w_i+\sum_{a, t(a)=i}v_{h(a)}$.
Consider the reduction $\A^{\theta_i}_{\lambda_i}(v,i):=D_R\red^{\theta_i}_{\lambda_i}\GL(V_i)$.
Since $i$ is a source and $\theta_i>0$,
this reduction is $$D^{\tilde{\lambda}_i}_{\operatorname{Gr}(v_i,\tilde{w}_i)}\otimes D_{\underline{R}},$$
where the notation is as follows. We write $\underline{R}$ for $$\bigoplus_{a, t(a)\neq i}\operatorname{Hom}(V_{t(a)},V_{h(a)})\oplus \bigoplus_{j\neq i}\operatorname{Hom}(V_j,\C^{w_j}),$$
where $w_j=\delta_{j0}$. So $R=\operatorname{Hom}(V_i, \C^{\tilde{w}_i})\oplus \underline{R}$.
The superscript $\tilde{\lambda}_i$ above indicates the  differential operators with coefficients in
the line bundle $\mathcal{O}(\tilde{\lambda}_i)$ on $\operatorname{Gr}(v_i,\tilde{w}_i)$, where  $\tilde{\lambda}_i:=\lambda_i-\tilde{w}_i/2$. The category $\A_\lambda^\theta(v)\operatorname{-mod}$
is the quotient of $D_R\red^{\theta_i}_{\lambda_i} \GL(V_i)\operatorname{-mod}^{\underline{G},\underline{\lambda}}$
by a Serre subcategory, the superscript $(\underline{G},\underline{\lambda})$ means ``$\underline{\lambda}$-twisted $\underline{G}$-equivariant
modules''. Here we write $\underline{G}$ for $\prod_{j\neq i}\GL(V_i), \underline{\lambda}:=(\tilde{\lambda}_j)_{j\neq i}$. The formula for
$\tilde{\lambda}_j$ is
$$\tilde{\lambda}_j:=\lambda_j-\frac{1}{2}(w_j+\sum_{t(a)=j}v_{h(a)}-\sum_{h(a)=j}v_{t(a)}).$$
Note that $\A^\theta_\lambda(v)$ is a quantum Hamiltonian reduction of $D(R)$ for the quantum comoment
map $x\mapsto x_R$ and level $\tilde{\lambda}=(\tilde{\lambda}_j)_{j\in Q_0}$.
The Serre subcategory in $\A_{\lambda_i}^{\theta_i}(v,i)\operatorname{-mod}^{\underline{G},\underline{\lambda}}$
we have to mod out consists of all modules whose singular support is contained in the
image of $\mu_i^{-1}(0)^{\theta_i-ss}\setminus \mu^{-1}(0)^{\theta-ss}$ in $T^*R\red^{\theta_i}\GL(V_i)$.

One can define functors
\begin{equation}\label{eq:funs} F_i:D^b_{{\underline{G},\underline{\lambda}}}(\A_{\lambda_i}^{\theta_i}(v,i)\operatorname{-mod})
\rightleftarrows
D^b_{\underline{G},\underline{\lambda}}(\A_{\lambda_i}^{\theta_i}(v+\epsilon_i,i)\operatorname{-mod}):E_i\end{equation}
in a standard way (pull-push), using the correspondence $\operatorname{Fl}(v_i,v_{i}+1,\tilde{w}_i)$
consisting of 2 step flags (or more precisely, $\operatorname{Fl}(v_i,v_{i}+1,\tilde{w}_i)\times
\delta_{\underline{R}}$, where $\delta_{\underline{R}}$ stands for the diagonal in $\underline{R}$).
Here we write $D^b_{{\underline{G},\underline{\lambda}}}(\bullet)$ for the $(\underline{G},\underline{\lambda})$-equivariant
derived category in the sense of Bernstein and Lunts. Of course, the functors $E_i,F_i$ have sense for the non-equivariant
derived categories as well.

The category $D^b(\A_\lambda^\theta(v)\operatorname{-mod})$
is the quotient of $D^b_{{\underline{G},\underline{\lambda}}}(\A_{\lambda_i}^{\theta_i}(v,i)\operatorname{-mod})$
by the full subcategory of all complexes whose homology satisfy the support condition mentioned above.

As Webster checked in \cite{Webster}, the functors $E_i,F_i$ descend to functors between the
quotient categories $D^b(\A_\lambda^\theta(v)\operatorname{-mod}), D^b(\A_{\lambda'}^\theta(v+\epsilon_i)\operatorname{-mod})$.
Here we write $\lambda'$ for the parameter producing the vector $\tilde{\lambda}$ but for the dimension
$v+\epsilon_i$, we have  $\lambda'_k=\lambda_k+n_{ik}/2$, where
$n_{ik}$ is the number of arrows connecting $i$ to $k$.

Let us write $R\Gamma_i$ for the derived global section functor $$D^b_{\underline{G},\underline{\lambda}}(\A_{\lambda_i}^{\theta_i}(v,i)\operatorname{-mod})
\rightarrow D^b_{\underline{G},\underline{\lambda}}(\A_{\lambda_i}(v,i)\operatorname{-mod}),$$ where
$\A_{\lambda_i}(v,i):=\Gamma(\A_{\lambda_i}^{\theta_i}(v,i))$. We have
$$\A_{\lambda_i}(v,i)=D^{\tilde{\lambda}_i}(\operatorname{Gr}(v_i,\tilde{w}_i))\otimes D(\underline{R}).$$

\begin{Lem}\label{Lem:part_glob_van}
Suppose $2v_i\leqslant \tilde{w}_i$ (and still $i$ is a source). If $|\lambda_i|\leqslant \tilde{w}_i/2-v_i$, then $R\Gamma_i$ annihilates every  $\A_{\lambda_i}^{\theta_i}(v,i)$-module $M$ that appears as a subquotient in
the homology of a complex lying in the image of $F_i$.
\end{Lem}
\begin{proof}

It is enough to prove this claim for the usual (non-equivariant) derived categories.

Let $M'\in \A_{\tilde{w}_i/2}^{\theta_i}(v,i)\operatorname{-mod}$ be the twist of $M$ by the line bundle
$\mathcal{O}(-\tilde{\lambda}_i)$ on  $\operatorname{Gr}(v_i,\tilde{w}_i)$.
It follows from the Beilinson-Bernstein theorem that $\Gamma_i: \A^{\theta_i}_{\tilde{w}_i/2}(v,i)\operatorname{-mod}
\rightarrow \A_{\tilde{w}_i/2}(v,i)\operatorname{-mod}$
is a category equivalence.

The functor $F_i: D^b(\A_{\tilde{w}_i/2}(v-\epsilon_i,i)\operatorname{-mod})\rightarrow
D^b(\A_{\tilde{w}_i/2}(v,i)\operatorname{-mod})$ can be realized as
a tensor product $\mathcal{B}\otimes^L_{D(\operatorname{Gr}(v_i-1,\tilde{w}_i))}\bullet$, where
$\mathcal{B}$ is a complex of $D(\operatorname{Gr}(v_i,\tilde{w}_i))$-$D(\operatorname{Gr}(v_i-1,\tilde{w}_i))$-
bimodules with Harish-Chandra homology (both algebras are quotients of $U(\mathfrak{gl}(\tilde{w}_i))$
and so the notion of a HC bimodule does make sense). Under our assumption on $v_i$, the Gelfand-Kirillov
dimension of $D(\operatorname{Gr}(v_i-1,\tilde{w}_i))$ is less than that of $D(\operatorname{Gr}(v_i,\tilde{w}_i))$.
It follows that there is a proper two-sided ideal $I$ in $D(\operatorname{Gr}(v_i,\tilde{w}_i))$
annihilating all Harish-Chandra $D(\operatorname{Gr}(v_i,\tilde{w}_i))$-$D(\operatorname{Gr}(v_i-1,\tilde{w}_i))$-bimodules
(there is a minimal nonzero two-sided ideal and we can take it for $I$).

Note that $R\Gamma_i(M)=\operatorname{RHom}_{A}(B,\Gamma_i(M'))$, where we write $A$
for $D(\operatorname{Gr}(v_i,\tilde{w}_i))$ and $B$ for the translation
$D(\operatorname{Gr}(v_i,\tilde{w}_i))$-$D^{\tilde{\lambda}_i}(\operatorname{Gr}(v_i,\tilde{w}_i))$-bimodule.
Then $$\operatorname{RHom}_{A}(B,\Gamma_i(M'))=\operatorname{RHom}_{A/I}(B/IB, \Gamma_i(M')).$$ So it remains
to check that $B/IB=0$.

Assume the contrary. The bimodule $B$ and the ideal $I$ are HC bimodules. Applying a suitable restriction
functor, see \cite[5.4]{BL}, to $B/IB$, we get a nonzero finite dimensional bimodule over the slice
algebras to $D^{\tilde{\lambda}_i}(\operatorname{Gr}(v_i,\tilde{w}_i)), D(\operatorname{Gr}(v_i,\tilde{w}_i))$.
Recall that a slice algebra to $D^{\lambda_i'}(v_i,\tilde{w}_i)$ for $\lambda'_i\in \C$
has the form $D^{\lambda_i'+k}(v_i-k, \tilde{w}_i-2k)$, this follows, for example, from \cite[5.4]{BL}.
From  $|\lambda_i|\leqslant \tilde{w}_i/2-v_i$, we see that the slice algebras for
$D^{\lambda_i}(\operatorname{Gr}(v_i,\tilde{w}_i))$
have no finite dimensional representations. This proves $B=IB$ and completes the proof of the lemma.
\end{proof}

The lemma has the following important corollary.

\begin{Cor}\label{Cor:glob_vanish}
Suppose $2v_i\leqslant \tilde{w}_i$ (and still $i$ is a source). If $|\lambda_i|\leqslant \tilde{w}_i/2-v_i$, then
$R\Gamma_i$ annihilates every simple $\A_{\lambda}^{\theta}(v)$-module $M_0$ that appears
in the homology of $F_i N_0$, where $N_0$ is a simple $\A_{\lambda}^{\theta}(v-\epsilon_i)$-module
supported on $\rho_{v-\epsilon_i}^{-1}(0)$.
\end{Cor}
\begin{proof}
Let $\pi^\theta(v):D(R)\operatorname{-mod}^{G,\tilde{\lambda}}\twoheadrightarrow
\A_\lambda^\theta(v)\operatorname{-mod}$ denote the quotient functor, it factors as $\pi^{\underline{\theta}}(v)\circ
\pi^{\theta_i}(v,i)$, where $\pi^{\theta_i}(v,i):D(R)\operatorname{-mod}^{G,\tilde{\lambda}}
\twoheadrightarrow \A^{\theta_i}_{\lambda_i}(v,i)\operatorname{-mod}^{\underline{G},\underline{\lambda}}\operatorname{-mod}$
and $\pi^{\underline{\theta}}(v):
\A^{\theta_i}_{\lambda_i}(v,i)\operatorname{-mod}^{\underline{G},\underline{\lambda}}\operatorname{-mod}
\rightarrow \A_\lambda^\theta(v)\operatorname{-mod}$ are the quotient functors. We also have derived analogs of these functors
(between the corresponding equivariant derived categories) to be denoted by the same letters.
The functors $\pi^\theta(v),\pi^{\underline{\theta}}(v),\pi^{\theta_i}(v,i)$ have derived right adjoints,
for example, $R\pi^\theta(v)^*$ lifts a complex of $\A_\lambda^\theta(v)$-modules to
that of $(G,\tilde{\lambda})$-equivariant $D_R|_{T^*R^{\theta-ss}}$-modules and then takes derived
global sections, compare with \cite[Lemma 4.6]{MN}. So we see that $R^j\Gamma(\bullet)=R^j\pi^\theta(v)^*(\bullet)^G$
for all $j$ (we pass to the cohomology because a priori the two derived functors
map to slightly different categories). So it is enough to prove that $R^j\pi^\theta(v)^*(M_0)^{\operatorname{GL}(v_i)}=0$.

From the description of $R\pi^\theta(v)^*$ (and similar descriptions of the other two derived functors),
we deduce that  $R\pi^\theta(v)^*(\bullet)^{\operatorname{GL}(v_i)}=R\Gamma_i\circ R\pi^{\underline{\theta}}(v)^*(\bullet)$.
So it suffices to show that $R\Gamma_i\circ R\pi^{\underline{\theta}}(v)^*(M_0)=0$.
By \cite[Corollary 6.2]{BL} and the main result
of \cite{BL}, $N_0$ is regular holonomic. By the proof of \cite[Corollary 6.2]{BL}, $F_i N_0$
is the direct sum of simple $\A_{\lambda_i}^{\theta_i}(v,i)$-modules with homological shifts. So it is sufficient to show
that $R\Gamma_i\circ R\pi^{\underline{\theta}}(v)^*\circ F_i(N_0)=0$.

Recall that $\pi^{\underline{\theta}}(v-\epsilon_i)\circ E_i\cong E_i\circ \pi^{\underline{\theta}}(v)$.
Also recall that (up to a homological shift) $F_i$ is right adjoint to $E_i$. So we have an isomorphism
$F_i\circ R\pi^{\underline{\theta}}(v-\epsilon_i)^*\cong  R\pi^{\underline{\theta}}(v)^*\circ F_i$
(up to a homological shift). So it suffices to show that $R\Gamma_i$ annihilates
$F_i\circ R\pi^{\underline{\theta}}(v-\epsilon_i)^*(N_0)$. But this follows from
Lemma \ref{Lem:part_glob_van}.
\end{proof}

\subsection{Absence of finite dimensional representations}
\begin{Lem}\label{Lem:no_fin_equiv}
Let $\theta$ be a generic stability condition. Then the following conditions are equivalent.
\begin{itemize}
\item[(a)] The algebra $\A_\lambda(v)$ has no finite dimensional representations.
\item[(b)] $R\Gamma(M)=0$ for any simple $\A_\lambda^\theta(v)$-module $M$
supported on $\rho_v^{-1}(0)$.
\end{itemize}
\end{Lem}
\begin{proof}
We can view $R\Gamma$ as a functor $D^-(\A_\lambda^\theta(v)\operatorname{-mod})\rightarrow D^-(\A_\lambda(v)\operatorname{-mod})$.
This functor has a left adjoint and right inverse functor, namely
the derived localization functor $$L\operatorname{Loc}: D^-(\A_\lambda(v)\operatorname{-mod})
\rightarrow D^-(\A_\lambda^\theta(v)\operatorname{-mod}).$$ The functors $R\Gamma, L\operatorname{Loc}$
restrict to the subcategories $D^-_{\rho^{-1}(0)}(\A_\lambda^\theta(v)\operatorname{-mod})$,
$D^-_{fin}(\A_\lambda(v)\operatorname{-mod})$. So $R\Gamma$ is a quotient functor
$$D^-_{\rho^{-1}(0)}(\A_\lambda^\theta(v)\operatorname{-mod})\twoheadrightarrow
D^-_{fin}(\A_\lambda(v)\operatorname{-mod}).$$
Our claim follows.
\end{proof}

Now take $\theta$ with all positive coordinates, it is generic. The corresponding system of simple roots
for $\a^\lambda$ (as defined in \cite[6.1]{BL}) consists of positive roots. Moreover, the condition
in (1) of Theorem \ref{Thm:cycl_main} is equivalent to $|\langle \lambda,\alpha\rangle|\leqslant
(\omega_0,\alpha)/2$
for all simple roots $\alpha$ of $\a^\lambda$.

In \cite[6.1]{BL} we have constructed endofunctors $E_\alpha, F_\alpha$ of
$\bigoplus_v D^b(\A_{\lambda_v}^\theta(v)\operatorname{-mod})$ that preserve
$\bigoplus_v D^b_{\rho_v^{-1}(0)}(\A_{\lambda_v}^\theta(v)\operatorname{-mod})$
(we write $\lambda_v$ instead of $\lambda$, because the parameter depends on $v$,
see the previous subsection).
The construction goes as follows. Pick an element $\sigma$ in the Weyl group
$W(Q)$ of $Q$. Then we have a Lusztig-Maffei-Nakajima (LMN) isomorphism
$\sigma_*: X^{\theta}(v)\xrightarrow{\sim} X^{\sigma\theta}(\sigma\cdot v)$,
where we write $\sigma\cdot v$ for the dimension vector corresponding to
the weight $\sigma\nu$. This isomorphism lifts to that of quantizations:
$\sigma_*:\A_\lambda^\theta(v)\xrightarrow{\sim}\A_{\sigma\lambda}^{\sigma\theta}(\sigma\cdot v)$,
see \cite[2.2]{BL} (note that here we use the symmetrized quantum comoment map so the
parameter transforms linearly). The categories $\A_\lambda^\theta(v)\operatorname{-mod},
\A_{\lambda}^{\theta'}(v)\operatorname{-mod}$ are naturally equivalent as long as
$\theta,\theta'$ lie in the same Weyl chamber for $\a^\lambda$, see \cite[6.1]{BL}.
Now pick a simple root $\alpha$ and choose $\theta'$ inside the Weyl chamber for
$\a^{\lambda}$ containing $\theta$ and such that $\ker\alpha$ is a wall for the
$\g(Q)$-chamber of $\theta'$. So there is $\sigma\in W(Q)$ such that
$\sigma\alpha$ is a simple root for $\g(Q)$ and all entries of $\sigma\theta'$ are positive.
Then we set  $F_\alpha= \sigma_*^{-1}\circ F_i\circ \sigma_*, E_\alpha= \sigma_*^{-1}\circ E_i\circ \sigma_*$.

\begin{Prop}\label{Prop:abs_fin_dim}
Suppose (1) of Theorem \ref{Thm:cycl_main} holds. Then the algebra $\A_\lambda(n\delta)$
has no finite dimensional irreducible representations.
\end{Prop}
\begin{proof}
We will need the following facts from \cite{BL}:
\begin{itemize}
\item[(i)] For fixed $\alpha$, the functors $E_\alpha, F_\alpha$ define an $\mathfrak{sl}_2$-action on
$\bigoplus_{v} K_0(\A_{\lambda_v}^\theta(v)\operatorname{-mod}_{\rho_v^{-1}(0)})$. The classes of simples
form a dual perfect basis.
\item[(ii)] The simples annihilated by all $E_\alpha$ live in dimensions $v$ corresponding
to extremal weights.
\end{itemize}
(i) is in the proof of \cite[Proposition 6.3]{BL}, while (ii) is a consequence of the main theorem
of \cite{BL}, Theorem 1.1 there. (i) allows to define crystal operators on the set of simples
in $\bigoplus_v \A_{\lambda_v}^\theta(v)\operatorname{-mod}$. It also shows that if a simple in
$\A_{\lambda_v}^\theta(v)\operatorname{-mod}_{fin}$ does not lie in the image of $\tilde{f}_\alpha$,
then it is annihilated by $\tilde{e}_\alpha$ and hence by $E_\alpha$.

Now let $\sigma$ be as in the construction of $E_\alpha, F_\alpha$ and let $v':=\sigma\cdot v$. The condition
$|\langle \lambda,\alpha\rangle|\leqslant \frac{(\omega_0,\alpha)}{2}$ is easily seen to be equivalent
to $|(\sigma\lambda)_i|=|\langle \sigma\lambda,\alpha_i\rangle|\leqslant \frac{1}{2}(\sigma \nu,\alpha_i)$.
The latter is nothing else but $\frac{1}{2}\tilde{w}_i'-v_i'$, where $\tilde{w}_i'$ is constructed from
$v'$. By Corollary \ref{Cor:glob_vanish}, the functor $R\Gamma$ annihilates all simple $\A_\lambda^\theta(n\delta)$-modules lying in the image of $\tilde{f}_\alpha$.

So, if $R\Gamma(M)\neq 0$, then $E_\alpha M=0$ for all simple roots $\alpha$ of $\a^\lambda$. However,
this is impossible by (ii): the weight $\omega_0-n\delta$ cannot be extremal for $\a^\lambda$.
\end{proof}

\subsection{Proof of the main theorem}
We need to show that, under  conditions (1) and (2) of Theorem \ref{Thm:cycl_main},
there are no finite dimensional $e_{W'}H_c(W')e_{W'}$-modules for any parabolic subgroup
$W'\subset W=G(\ell,1,n)$. The parabolic subgroups of $W$ have the form $G(\ell,1,n_0)\times \mathfrak{S}_{n_1}
\times\ldots\times \mathfrak{S}_{n_k}$, where $n_0+n_1+\ldots+n_k\leqslant n$. The slice algebras are of the
form  $\A_\lambda(n_0\delta)\otimes \A_{\kappa}(n_1)\ldots\otimes \A_{\kappa}(n_k)$,
where we write $\A_{\kappa}(n_i)$ for the spherical subalgebra in $H_{\kappa}(\mathfrak{S}_{n_i})$.
This follows from \cite[5.4]{BL}. If (1) holds, then $\A_\lambda(n_0\delta)$ has no finite dimensional
irreducible representations (when $n_0>0$). If (2) holds, then the algebras $\A_{\kappa}(n_i)$ do not.
This completes the proof of the theorem.

\begin{Rem}\label{Rem:gener}
One can state a sufficient condition for the algebra $\A_\lambda(v)$ to be simple generalizing Theorem \ref{Thm:cycl_main}
for an arbitrary quiver of finite or affine type thanks to results of \cite{BL},\cite{Gies}. We are not going
to do this explicitly.
\end{Rem}

\end{document}